\documentclass[12pt,reqno]{article}
\usepackage{graphicx}
\usepackage{amssymb}
\usepackage{amsmath}
\usepackage{amsthm}
\usepackage{amsfonts}
\usepackage{url}
\usepackage{hyperref}
\usepackage{breakurl}
\usepackage[T1]{fontenc}        % For \DJ (Djokovi\'c)

\newcommand{\comment}[1]{}

\newtheorem{theorem}{Theorem}

\newtheorem{lemma}{Lemma}

\begin{document}
\bibliographystyle{plain}
\title{~\\[-40pt]
Note on a double binomial sum\\ 
relevant to the Hadamard\\
maximal determinant problem}
\author{Richard P.\ Brent\\
Australian National University\\
Canberra, ACT 0200,
Australia\\
\and
Judy-anne H.\ Osborn\\
The University of Newcastle\\
Callaghan, NSW 2308,
Australia\\
}

\date{\today}

\maketitle
\thispagestyle{empty}                   % To avoid page number

\begin{abstract}
We prove a double binomial sum identity 
\[ \sum_p \sum_q \binom{2k}{k+p} \binom{2k}{k+q}\, |p^2-q^2| =
 2k^2{\binom{2k}{k}}^2 \]
which differs from most binomial sum identities in that the summands involve
the absolute value function. The identity is of interest because it can be
used in proofs of lower bounds for the Hadamard maximal determinant problem.
Our proof of the identity uses a two-variable variant of the method of
telescoping sums.
\end{abstract}

\pagebreak[3]
\section{Introduction}		\label{sec:intro}

In this note we prove two results, one known and one new, involving binomial
sums where the absolute value function occurs in the summands.

Lemma~\ref{lemma:Best} is a binomial sum which has appeared
several times in the literature, e.g.~Alon and Spencer~\cite[\S2.5]{AS},
Best~\cite[proof of Theorem 3]{Best}, Brown and Spencer~\cite{BS},
Erd\H{o}s and Spencer~\cite[proof of Theorem 15.2]{ES}.
It was also a problem in the 
1974 Putnam competition~\cite[Problem A4]{Putnam74}.
Lemma~\ref{lemma:Best} can be used
to calculate the mean of each diagonal term that arises
when the probabilistic method is used to give lower bounds for
the Hadamard maximal determinant problem, as in~\cite{rpb253}.

Our new result is Theorem~\ref{thm:double_sum}, which gives a closed-form
expression for a double sum which is analogous to the single sum of
Lemma~\ref{lemma:Best}. Theorem~\ref{thm:double_sum} can be used to
calculate the second moment of each diagonal term that arises when the
probabilistic method is used to give lower bounds for the Hadamard maximal
determinant problem.  In~\cite[Theorems 2--3]{rpb253} we gave lower bounds
without using this second moment, but such results can be improved if
the second moment is known~\cite{rpbxxx}.

To prove Theorem~\ref{thm:double_sum} we first use symmetry to eliminate the
absolute value function occurring in the sum, and then use a
two-variable variant of the method of telescoping sums, where the
double sums collapse to give single sums which can be
evaluated explicitly. The ``magic'' cancellation giving the final result
suggests that a simpler proof could exist, but we have not found such
a proof.

\subsection*{Notation}

The variables $k$, $n$, $p$, $q$ denote integers
(not necessarily positive).

The binomial coefficient $\binom{n}{k}$ is defined to be zero if $k < 0$
or $k > n$. Using this convention, 
we can often avoid explicitly 
specifying upper and lower limits of sums involving binomial coefficients.  

\pagebreak[3]
\section{A well-known single sum}

Lemma~\ref{lemma:Best} is well-known, as noted
in the Introduction. We give a proof because it illustrates some of the
ideas used in the proof of Theorem~\ref{thm:double_sum}.

\begin{lemma}	\label{lemma:Best}
For all $k \ge 0$,
\[
\sum_p \binom{2k}{k+p}\,|p| = k\binom{2k}{k}\,. % magmarpb/t7
\]
\end{lemma}

\begin{proof}
Let $S_0 := \sum_p \binom{2k}{k+p}\,|p|$. Splitting the sum into sums
over positive and negative $p$, 
using $\binom{2k}{k+p} = \binom{2k}{k-p}$,
and observing that the term for $p=0$ vanishes, 
we see that
\begin{equation}
S_0 = 2\sum_{p>0}p\,\binom{2k}{k+p}\,.	\label{eq:Sigma_1}
\end{equation}
Writing $p = (k+p) - k$ gives
\begin{eqnarray*}
p\binom{2k}{k+p} &=& (k+p)\binom{2k}{k+p} - k\binom{2k}{k+p}\cr
	&=& 2k\binom{2k-1}{k+p-1} - k\binom{2k}{k+p}\,.
\end{eqnarray*}
Substituting this into~\eqref{eq:Sigma_1} and using
\[2\sum_{p>0}\binom{2k-1}{k+p-1}
 = \sum_p \binom{2k-1}{p} = 2^{2k-1}\]
and
\[2\sum_{p>0}\binom{2k}{k+p}
 = \sum_{p \ne 0} \binom{2k}{k+p}
 = 2^{2k} - \binom{2k}{k}
\]
gives
\[S_0 = k2^{2k} - k2^{2k} + k\binom{2k}{k}
 = k\binom{2k}{k}\,.\]
This completes the proof of Lemma~\ref{lemma:Best}.
\end{proof}

\section{The main result~-- a double sum}

\begin{theorem}		\label{thm:double_sum}
For all $k \ge 0$, % h = 4*k is possible Hadamard order
\[
\sum_p \sum_q \binom{2k}{k+p} \binom{2k}{k+q}\, |p^2-q^2| =
 2k^2{\binom{2k}{k}}^2\,. % magmarpb/t8
\]
\end{theorem}
\begin{proof}	% See magmarpb/check_second_moment_proof for formula check
Write 
\[S_1 := 
\sum_p \sum_q \binom{2k}{k+p} \binom{2k}{k+q}\, |p^2-q^2|\,.\]
The terms in $S_1$ for which $p=0$ or $q=0$ are
\begin{eqnarray*}
S_2 &:=& \sum_p \binom{2k}{k+p} \binom{2k}{k}\,p^2 +
	\sum_q \binom{2k}{k} \binom{2k}{k+q}\,q^2\\
    &=& 2\,\binom{2k}{k}\sum_{p}\binom{2k}{k+p}\,p^2\,,
\end{eqnarray*}
and the other terms are
\[S_3 := \sum_{p\ne 0} \sum_{q\ne 0} 
	\binom{2k}{k+p} \binom{2k}{k+q}\, |p^2-q^2|\,.\]
Considering the generating function
\[
  f(x) := \sum_p \binom{2k}{k+p}x^p
	= x^{-k}(1+x)^{2k} = (x^{1/2} + x^{-1/2})^{2k}\,,
\]
applying the operator $x{\rm d}/{\rm d}x$ twice, and evaluating
the resulting expression at $x=1$, we see that
\[
\sum_{p}\binom{2k}{k+p}\,p^2 = k2^{2k-1}\,,
\]
so 
\begin{equation}						\label{eq:S2}
S_2 = k2^{2k}\binom{2k}{k}\,.
\end{equation}
Also, using symmetry, we have
\[
S_3 = 8 \sum_{p>q>0} \binom{2k}{k+p}\binom{2k}{k+q}\,(p^2-q^2)\,.
\]
Now write 
$p^2-q^2 = (p-k)(p+k) - (q-k)(q+k)$,
so
\[
\binom{2k}{k+p}\binom{2k}{k+q}\,(p^2-q^2) =
\phantom{\binom{2k-2}{k+p-1}\binom{2k}{k+q}
	 \binom{2k-2}{k+p-1}\binom{2k}{k+q}}
\]
\[
2k(2k-1)\left[-\binom{2k-2}{k+p-1}\binom{2k}{k+q}
	      +\binom{2k}{k+p}\binom{2k-2}{k+q-1}\right]\,.
\]
On the right-hand side we use
\[\binom{2k}{k+p} = \binom{2k-2}{k+p} 
	+ 2\binom{2k-2}{k+p-1} + \binom{2k-2}{k+p-2}\]
and the corresponding identity for $\binom{2k}{k+q}$.
After cancelling the terms involving
$\binom{2k-2}{k+p-1}\binom{2k-2}{k+q-1}$ and regrouping, this gives
\begin{eqnarray*}
\frac{S_3}{8} &=& 2k(2k-1) \times\\
&& \left\{
\left[
\sum_{p>q>0}\binom{2k-2}{k+p}\binom{2k-2}{k+q-1}
-\sum_{p>q>0}\binom{2k-2}{k+p-1}\binom{2k-2}{k+q-2}
\right]\right. \\
&+& \left. \left[
\sum_{p>q>0}\binom{2k-2}{k+p-2}\binom{2k-2}{k+q-1}
-\sum_{p>q>0}\binom{2k-2}{k+p-1}\binom{2k-2}{k+q}
\right]
\right\}\,.\\
\end{eqnarray*}
Changing variables in the second and third summations, we obtain
\begin{eqnarray*}
\frac{S_3}{8} &=& 2k(2k-1) \times\\
&& \left\{
\left[
\sum_{p>q>0}\binom{2k-2}{k+p}\binom{2k-2}{k+q-1}
-\sum_{p>q\ge 0}\binom{2k-2}{k+p}\binom{2k-2}{k+q-1}
\right]\right. \\
&+& \left. \left[
\sum_{p>q \ge 0}\binom{2k-2}{k+p-1}\binom{2k-2}{k+q}
-\sum_{p>q>0}\binom{2k-2}{k+p-1}\binom{2k-2}{k+q}
\right]
\right\}\,.\\
\end{eqnarray*}
The expressions inside each pair of square brackets 
both involve a kind of two-variable
telescoping sum~-- the only terms that do not cancel are those
for $q=0$.  Thus, we obtain
\begin{eqnarray*}
\frac{S_3}{8} &=& 2k(2k-1) \times\\
&& \left\{
\sum_{p>0}\binom{2k-2}{k+p-1}\binom{2k-2}{k}
-\sum_{p>0}\binom{2k-2}{k+p}\binom{2k-2}{k-1}
\right\}\,.\\
\end{eqnarray*}
Taking out the factors that are independent of $p$, we are left
with two easily-evaluated sums, giving
\begin{eqnarray}
S_3 = 8k(2k-1) \!\!\!&\times&\!\!\!
\left\{
\binom{2k-2}{k}\left[2^{2k-2} - \binom{2k-2}{k-1}\right]
- \right.  \nonumber \\
&& \left.  \,\binom{2k-2}{k-1}\left[2^{2k-2}
-2\,\binom{2k-2}{k} - \binom{2k-2}{k-1}\right]
\right\}. \label{eq:S3}
\end{eqnarray}
Recall that $S_1 = S_2 + S_3$. Using the expressions~\eqref{eq:S2} for
$S_2$ and~\eqref{eq:S3} for $S_3$, and simplifying, 
we obtain the desired result for $S_1$.
\end{proof}

\subsection*{Acknowledgement}

We thank Warren D.\ Smith for his comments on earlier drafts of this note.


\begin{thebibliography}{99}
 
{%\small
\bibitem{AS}
N. Alon and J. H. Spencer, % Noga Alon, Joel H. Spencer
\emph{The Probabilistic Method}, 3rd edn., Wiley, 2008.

\bibitem{Putnam74}
Anonymous, % William Lowell 
Putnam %Mathematical 
Competition, 1974,
\url{http://www.math-olympiad.com/35th-putnam-mathematical-competition-1974-problems.htm}.

\bibitem{Best}
M. R. Best, The excess of a Hadamard matrix,
\emph{Nederl.\ Akad.\ Wetensch.\ Proc.\ Ser. A} \textbf{80}
$=$ 
\emph{Indag.\ Math.\ }\textbf{39} (1977), %no. 5, 
357--361.

\bibitem{rpb253}
R. P. Brent, J. H. Osborn and W. D. Smith,
\emph{Lower bounds on maximal determinants of $\pm1$ matrices via
the probabilistic method}. arXiv:1211.3248v3, 5 May 2013.

\bibitem{rpbxxx}
R. P. Brent, J. H. Osborn and W. D. Smith,
\emph{Improved lower bounds on maximal determinants of $\pm1$ matrices via
the probabilistic method}, in preparation.

\bibitem{BS}
T. A. Brown and J. H. Spencer,
Minimization of $\pm1$ matrices under line shifts,
\emph{Colloq.\ Math.} % (Poland) 
\textbf{23} (1971), 165--171.
Erratum \emph{ibid} pg.~177.

\bibitem{ES}
P. Erd\H{o}s and J. Spencer,  
\emph{Probabilistic Methods in Combinatorics},
Akad\'emiai Kiad\'o, Budapest, 1974.
Also published by Academic Press, New York, 1974.
} %\small

\end{thebibliography}
\end{document}